\documentclass[twoside]{asl}

\usepackage{mathrsfs}
\usepackage{amssymb}
\usepackage{braket}
\usepackage{verbatim}
\usepackage{ stmaryrd }

\newtheorem{thm}{Theorem}
\newtheorem{prop}[thm]{Proposition}

\newtheorem{lemma}[thm]{Lemma}

\theoremstyle{definition}
\newtheorem{defin}[thm]{Definition}

\makeatletter
\newtheorem*{rep@theorem}{\rep@title}
\newcommand{\newreptheorem}[2]{%
\newenvironment{rep#1}[1]{%
 \def\rep@title{#2 \ref{##1}}%
 \begin{rep@theorem}}%
 {\end{rep@theorem}}}
\makeatother

\newreptheorem{theorem}{Theorem}

\newcommand\power{\mathop\mathcal P}
\newcommand\R{r}
\newcommand\G{S_{\power(\Omega)}}
\newcommand\Q{q}
\newcommand\Aut{\mathop{\text{Aut}}\nolimits}
\newcommand\imp\rightarrow
\newcommand\ekv\leftrightarrow
\newcommand{\Fix}{\mathop{\text{Fix}}\nolimits}
\newcommand{\Gal}{\mathop{\text{Gal}}\nolimits}

\newcommand{\Inv}{\mathop{\text{Inv}}\nolimits}
\newcommand{\Sim}{\mathop{\text{Sim}}\nolimits}
\newcommand{\liim}{\ensuremath{\mathop{\La_{\infty\infty}^-}}}
\newcommand{\lii}{\ensuremath{\mathop{\La_{\infty\infty}}}}
\newcommand{\loo}{\ensuremath{\mathop{\La_{\omega\omega}}}}
\newcommand{\lio}{\ensuremath{\mathop{\La_{\infty\omega}}}}

\newcommand{\La}{\mathscr{L}}

\let\oldmarginpar\marginpar\renewcommand\marginpar[1]{\-\oldmarginpar[\raggedleft\footnotesize
#1]%
{\raggedright\footnotesize #1}}

\title[Invariance and definability]{Invariance and definability, with and without equality}
\author{Denis Bonnay}
\revauthor{Bonnay, Denis}
\address{D\'epartement de Philosophie\\
Universit\'e Paris Ouest Nanterre\\
200, avenue de la R\'epublique\\
92001 Nanterre CEDEX\\
FRANCE}
\email{denis.bonnay@u-paris10.fr}
\author{Fredrik Engstr\"om}
\revauthor{Engstr\"om, Fredrik}
\address{Department of Philosophy, Linguistics and Theory of Science \\ Gothenburg University \\ Box
200\\ 405 30 Gothenburg\\ SWEDEN}
\email{fredrik.engstrom@gu.se}
\thanks{Both authors were partially supported by the EUROCORE LogICCC LINT program and the
Swedish Research Council.}

\begin{document}
\begin{abstract}
The dual character of invariance under transformations and definability by some operations has been
used in classical work by for example Galois and Klein. Following Tarski, philosophers of logic have claimed
that logical notions themselves could be characterized in terms of invariance. In this paper, we generalize a correspondence due to Krasner between invariance under groups of permutations and definability in $\La_{\infty\infty}$ so as to cover the cases (quantifiers, logics without equality) that are of interest in the logicality debates, getting McGee's theorem about quantifiers invariant under all permutations and definability in pure $\La_{\infty\infty}$ as a particular case. We also prove some optimality results along the way, regarding the kind of relations which are needed so that every subgroup of the full permutation group is characterizable as a group of automorphisms.

\end{abstract}

\maketitle

\section{Introduction}

Permutations and relations may be seen as two ways of encoding information about a domain of objects. In the case of relations, it is pretty obvious that some piece of information is encoded: a difference is made between objects between which the relation holds, and objects between which it does not.
This is also true, if slightly less obvious, of permutations. By taking into account a permutation,
it is recognized that some differences between objects do not matter: given the purpose at hand, one
may as well consider that such and such objects have been interchanged.


Moreover, permutations and relations may be seen as two \textit{dual} ways of encoding information about
objects. Permutations tell us which objects are similar (they can be mapped to each other), relations tell us which objects are
different (think of a unary relation true of one object and false of another). This duality is at work whenever mathematicians or physicists consider groups of isomorphisms and classes of invariants.\footnote{Klein's Erlanger Program \cite{Klein:1872} and the
systematic study of the duality between groups of transformations on a space and geometries on that
space is a famous example. It was also a direct inspiration to Tarski \cite{Tarski:86}.}
A group of automorphisms is a group of permutations that respect a given structure: permutations are extracted
from a domain equipped with relations. A class of invariants consists of a class of relations that
respect a given group of permutations: extensions for symbols to be interpreted in the domain of objects
under consideration are extracted from permutations.

In logic, this duality has been extensively investigated within model theory. Typical questions include the characterization of structures that can be described in terms of their automorphism groups, or more generally how much automorphism groups say about the structures they come from (see \cite{Cameron:1990} on model theory and groups). Within the philosophy of logic, the duality between permutations and extensions comes to the forefront in debates regarding the characterization of logical constants. Invariance under all permutations has been taken as the formal output of a conceptual analysis of what it is to be a logical notion. Thus, it has been proposed as a mathematical counterpart to the generality of logic (by Tarski himself, \cite{Tarski:86}) and to its purely formal nature (by Sher \cite{Sher:91} and MacFarlane \cite{MacFarlane:00}).

The aim of the present paper is to bring these two traditions closer together by adapting some results in model theory to the cases that are discussed by philosophers of logic interested in invariance as a logicality criterion. As a case in point, we will show how the much quoted theorem by McGee in \cite{McGee:96} about invariance under all permutations and definability in the pure infinitary logic \lii~can be seen as following from a suitable extension of much earlier work by Krasner (\cite{Krasner:38}, \cite{Krasner:50}). Our motivation for such a generalization is that a conceptual assessment of a given logicality criterion stated in terms of invariance properties should be grounded in a prior evaluation of the significance of the general duality between groups of permutations and sets of relations and in an independent assessment of what invariance under groups of permutations (or other kind functions) can and cannot do in the contexts of interest. From a technical perspective, our starting point will be Krasner's work and his so-called abstract Galois theory. Application to cases of interest in the logicality debates involve two generalizations. In order to cover the interpretation of quantifiers, we will to need to consider structures equipped with second-order relations. In order to cover scenarii in which the logicality of identity is not presupposed, we will to need to consider functions that are not injective, as suggested by Feferman \cite{Feferman:99}.

Our main result (Theorem \ref{cor} below) provides a unified perspective on the duality by establishing a correspondence between certain monoids of relations and sets of operations closed under definability in an infinitary language without equality. In Section \ref{section:krasner}, we recall Krasner's correspondence between classes of relations of possibly infinite arity closed under definability in \lii~and subgroups of the permutation group. We show how Krasner's correspondence relates to the standard model-theoretic characterization of groups of permutations which are the automorphism groups of a first-order structure (with relations of finite arity) as closed groups in the topology of point-wise convergence. Section \ref{section:extending} is devoted to extending Krasner's correspondence to second-order relations. We prove some optimality results along the way: what kind of relations are really needed so that every subgroup is the automorphism group of a class of relations? In the first-order case, infinitary relations were necessary. Second-order binary relations suffice to do the job, but second-order unary relations do not. The next section is devoted to relaxing the assumption of injectivity, extending the correspondence to logics without equality. Finally, we draw some lessons in Section \ref{section:final} for the logicality debate.

\section{Abstract Galois theory and permutation groups} \label{section:krasner}

\subsection{Krasner's abstract Galois theory}

To begin with, we shall recall the fundamentals of Krasner's \textit{abstract} Galois theory,
following in particular Poizat \cite{Poizat:85}. The starting point of Krasner is \textit{classical}
Galois theory, in which a particular instance of the general duality we have outlined is at play.
Classical Galois theory is the study of the duality between the subgroups of the automorphism group
of certain types of field extensions, called Galois extensions, and the intermediate fields in the
field extension.

A field is an algebraic structure with addition, multiplication, inverses and units. If $k \subseteq
K$ are both fields then we say that $K:k$ is a field extension; it is Galois if it is algebraic,
normal, and separable (definitions may be found in \cite{Lang:02}). Given a field extension $K:k$
let $G$ be the Galois group of $K:k$, i.e., the group of automorphisms of $K$ fixing the elements of
$k$ point-wise. Define the following pair of mappings:

\begin{align*}
  \Fix(H)&=\set{a \in K| ha=a \text{ for all $h \in H$}} \\
  \Gal(A) &= \set{g \in G | ga=a \text{ for all $a \in A$}}
\end{align*}
where $H$ is any subset of $G$ and $A$ any subset of the domain of $K$. The fundamental theorem of
Galois theory says that, for finite Galois extensions, $\Fix(\Gal(A))$ is the smallest subfield of
$K$ including $k \cup A$ and that $\Gal(\Fix(H))$ is the smallest subgroup of $G$ including $H$.
There is thus a one-to-one correspondence between fields $K'$ such that $k \subseteq K' \subseteq K$
and subgroups of $G$.

According to Krasner, ``the true origin of Galois theory [does not lie] in algebra, in the strict
sense of the word, but in logic'' (\cite{Krasner:50}, p. 163). This claim is to be substantiated by
exhibiting a general duality, between automorphism groups and relations definable over a structure.
Consider a domain $\Omega$, $S_\Omega$ the full symmetric group on $\Omega$ and a set $\R$ of relations
on $\Omega$; these relations may include infinite arity relations regarded as subsets of
$\Omega^\alpha$ where $\alpha$ is some ordinal number. We now define the following pair of mappings:

\begin{align*}
  \Inv(H) &= \set{R \subseteq \Omega^{\alpha} | hR=R \text{ for all $h \in H$},
  \alpha \leq |\Omega|}  \\
\Aut(\R) &= \set{g \in S_\Omega | gR=R \text{ for all $R \in \R$}}
\end{align*}

\noindent The logic $\lii$ is the infinitary generalization of the predicate calculus where formulas
are built by means of arbitrarily long conjunctions and disjunctions and by means of arbitrarily
long universal and existential quantifiers sequences (see e.g. \cite{Karp:1964}).
The $\lii$-closure of $\R$ is the set of relations definable in $\lii$ from relations in $\R$. The
following Theorem is essentially\footnote{The qualification `essentially' is due to the fact that
Krasner thinks directly in terms of closure under `logical' operations on relations rather in terms
of definability in a formal language.} shown by Krasner in \cite{Krasner:38}:

\begin{thm}\label{thm:kras1} Let $\Omega$ and $S_\Omega$ be as above, $H \subseteq S_\Omega$ any set of
  permutations and $\R$ any set of relations on $\Omega$ of arities at most $|\Omega|$.
  \begin{enumerate}
    \item $\Inv(\Aut(\R))$ is the $\La_{\infty\infty}$-closure of $\R$.
    \item $\Aut(\Inv(H))$ is the smallest subgroup of $S_\Omega$ including $H$.

  \end{enumerate}
\end{thm}

\noindent Thus, there is a one-to-one correspondence between the subgroups of the full
symmetric group $S_\Omega$ of $\Omega$ and the sets of relations closed under definability in
$\La_{\infty\infty}$.

How does this relate with the original correspondence of classical Galois theory? The first half of
the fundamental theorem of Galois Theory can be derived from the first half of Theorem
\ref{thm:kras1}.\footnote{To our knowledge, the question regarding the other halves is still open.}
Given a field extension $K:k$ and $A \subseteq K$, let $\R$ be the set of addition, multiplication,
inverse and the relation $R_a=\set{a}$ for every $a \in k\cup A$. By definition, we have that
$\Gal(A)=\Aut(\R)$. Now $\Inv(\Aut(\R))$ and $\Fix(\Gal(A))$ are not exactly on a par, because
Krasner's theory takes sets of {\em relations} as values instead of, as in the classical Galois
theory, sets of {\em points}. However, the set of fixed points can be extracted from the set of
fixed relations by observing that an
element $a \in \Omega$ is fixed by a permutation iff the singleton relation $\set{a}$ is fixed.
Given that $\Inv(\Aut(\R))$ is closed under definability in $\La_{\infty\infty}$, it is immediate
that this set of fixed points is a subfield of $K$. In order to get the first half of the
fundamental Galois theorem, one also needs to show that it is included in any subfield extending $k
\cup A$. Krasner proves this in \cite{Krasner:50} by showing that if $R_a$ is definable in $\lii$
over $\R$, then $a$ can be reached from elements in $k \cup A$ by means of the field
operations.\footnote{This amounts to a property of eliminability of quantifiers. Krasner gives a
general characterization of structures having this property. See Poizat \cite{Poizat:85} regarding
connections with contemporary model theory.}

\subsection{Permutation groups and topology}

Krasner's result shows that for any group of permutations of $\Omega$ there is some set of relations $\R$ on $\Omega$ such that $\Aut(\R)$ is exactly that group.
This is not true if we restrict the relations to be of finite arity.
If $\R$ is a set of relations of finite arity then $(\Omega,\R)$ is a `standard' first-order
relational structure and, obviously,
the group $\Aut(\R)$ coincides with the automorphism group of the first-order structure
$(\Omega,\R)$.
However it is not true that every subgroup of the full permutation group is the automorphism group
of a
first-order structure (with only finite arity relations). We will give a detailed proof of the well
known characterization of autormophism groups of first-order structures in terms of topological
properties. This is done partly to set up the scene for the next section (but see also
\cite{Cameron:1990}).

First we need to introduce the topology, the product topology, on the group of permutations that we
will be
using.
Let $\Omega$ be any infinite set equipped with the discrete topology, i.e., the topology in which
every set
is open. Let $G$ be the full symmetric group $S_\Omega$ on $\Omega$, i.e., the group of all permutations of $\Omega$.
The stabilizer $G_{\bar a}$ of a finite tuple
$\bar a \in \Omega^k$ is the set of all permutations in $G$ keeping all elements in $\bar
a$ fixed, i.e., $$G_{\bar a}=\set{g \in G | g\bar a=\bar a}.$$

The symmetric group $S_\Omega$ on $\Omega$ is a subset of $\Omega^\Omega$; let $G$ inherit the product
topology from
$\Omega^\Omega$, i.e., a basis for this topology are the (right or left) cosets of stabilizers
of finite tuples, that is, sets of the following form:
$$G_{\bar a,\bar b}=\set{g \in G| g\bar a=\bar b},$$
where $\bar a, \bar b \in \Omega^k$.\footnote{This topology makes $G$ into a topological group which
means that
the operations of multiplication (in this case composition) and taking inverses are
continuous. As with topological groups in general, all open subgroups are also closed: let $H$ be
an open subgroup of $G$ then the complement of $H$ is the union of all
cosets $gH$ of $H$ which are disjoint from $H$, and thus the complement of $H$ is
open. }

This is the topology of pointwise convergence: The (topological) closure of a set $A \subseteq G$ consists of all
permutations $g \in G$ such that for every $k$ and every $\bar a \in \Omega^k$ there is $f \in A$
such that
$f(\bar a)=g(\bar a)$.\footnote{If $\Omega=\set{a_0,a_1,\ldots}$ is countable then the topology is
completely
metrizable by the metric $d(g,h)=1/2^k$ if $g \neq h$ and $k$ is the least number such that
$ga_k\neq
ha_k$ or $g^{-1}a_k \neq h^{-1}a_k$, and $d(g,g)=0$. Thus, a sequence of permutations
$g_0,g_1,\ldots$ converges to $g$ if and only if, for any $k$, $g_na_k=ga_k$ for all sufficiently
large $n$.
It should also be noted that $G$ is not a closed subset of $\Omega^\Omega$ which
means that there are sequences $g_i$ of permutations of $\Omega$ converging
to a function $f$ which is not a permutation: Let $\Omega=\mathbb N$, the sequence
$$(0,1), (0,1,2), (0,1,2,3), \ldots$$
converges to the function $f:n\mapsto n+1$ which is not a permutation.}

Given a subgroup $H$ of $G$ let $O^H_{\bar a}=\set{h\bar a | h \in H}$ be the
orbit of $\bar a$ under $H$. Orbits provide a canonical way of `translating' a
group of permutations on a domain into relations on that domain. $(\Omega,O^H_{\bar a})_{\bar a
\in \Omega^k}$ is
thus a first-order relational structure, called the canonical structure of $H$.

The following well-known folklore result characterizes the subgroups of $G$ which are automorphism groups of a first-order
structure (with only finite arity relations):

\begin{prop}\label{prop:3}
  Let $H$ be a subgroup of $G$. The following are equivalent:
  \begin{enumerate}
    \item\label{structure} There is a first-order structure $M$ with domain $\Omega$ whose
      automorphism group is $H$.
\item\label{canonical} $H$ is the automorphism group of the canonical structure of $H$, $(\Omega,O^H_{\bar a})_{\bar a
\in \Omega^k}$.
    \item\label{closed} $H$ is closed.
  \end{enumerate}
\end{prop}
\begin{proof}
\eqref{structure} implies \eqref{closed}: Let $g \in G$ be in the closure of $H$, i.e., for any
$\bar a \in \Omega^k$ there is $h \in \Aut(M)$ which is identical to $g$ on $\bar a$. Let $R$ be a
  relation in $M$, $\bar a \in \Omega^k$ and $h \in \Aut(M)$ such that $g\bar a=h\bar a$,
  then $R\bar a$ iff $Rh\bar a$ iff $Rg\bar a$ and so
$R$ is invariant under $g$. Similarly we can show that any relation, constant and function in $M$ is
invariant under $g$; thus $g \in \Aut(M)$; and $\Aut(M)=H$ is closed.

  \eqref{canonical} implies \eqref{structure}: Obvious.

  \eqref{closed} implies \eqref{canonical}: Let $M$ be the canonical structure of
  $H$. Clearly $H \subseteq \Aut(M)$ since $H$ respects
  all its orbits, i.e., for all $h \in H$, $\bar a \in O^H_{\bar b}$ iff $h\bar a
  \in O^H_{\bar b}$. Assume $g \in \Aut(M)$ and $\bar a \in \Omega^k$, then
  $g\bar a \in O^H_{\bar a}$ and so there is an $h \in H$ such that $g\bar a=h\bar a$.
This shows that $g$ is in the closure of $H$; but since $H$ is closed we have
  $\Aut(M) \subseteq H$, proving that $H$ is the automorphism group of $M$.
\end{proof}

The proof actually shows that for any subgroup $H$ of $G$ the
automorphism group of the canonical structure of $H$ is the closure of $H$.

In order to see that not all subgroups of $G$ are automorphism groups of first-order
structures, it is, using the proposition, sufficient to find a subgroup of $G$ which is not closed.
Let the support of a permutation $g$ of $\Omega$ be the set of points moved by $g$, i.e., the set
$\set{a \in \Omega | ga \neq a}$. The set $H$ of permutations with finite
support is a non-closed subgroup of $G$: Take any permutation $g \in G$ and
any $\bar a \in \Omega^k$ we can find $h \in H$ such that $g\bar a = h\bar a$ by
extending the bijection $g|\bar a: \set{\bar a} \to \set{g\bar a}$ to a permutation of
$A=\set{\bar a} \cup \set{g\bar a}$. Let then $h$ be this permutation on $A$ and the identity
function outside $A$. This shows that the closure of $H$ is the full group $G$, and since not every
permutation has a finite support, $H$ is not closed.

Hence the use of relations of infinite arity in Krasner's correspondence is indeed necessary: For
non-closed $H$, if $H=\Aut(M)$ then $M$ includes relations of infinite arity.

There is a connection between the allowed arities of relations in the structures and the topological
closure property characterizing automorphism groups as the following proposition shows.
To see this, let us define the $\alpha$-topology, for $\alpha$ an infinite ordinal, on $G$ in the
following manner. The basic open sets are of the form
$$G_{\bar a,\bar b}\set{g \in G| g\bar a=\bar b},$$
where $|\bar a|=|\bar b| < \alpha$. The $\omega$-topology
coincides with the ordinary product topology described above.
It should be noted that the
$\beta$-topology is a refinement of the $\alpha$-topology if $\alpha\leq\beta$,
which means that an open set in the $\alpha$-topology is also open in the
$\beta$-topology. If we relax the definition of a first-order structure to include
relations of arity less than $\alpha$ and define $\alpha$-orbits of a subgroup $H$ to be
$$O^H_{\bar a}=\set{h\bar a | h \in H},$$ where now $\bar a \in \Omega^\alpha$;
then the proof of Proposition \ref{prop:3}
goes through with the canonical structure of $H$ replaced by $(\Omega, O^H_{\bar a})_{\bar a \in
\Omega^\alpha}$
 and `closed' replaced by
`closed in the $\alpha$-topology,' giving the following generalization:

\begin{prop}\label{prop:generalizedtopology}
  Let $H$ be a subgroup of $G$ and $\alpha$ an infinite ordinal, then the following are equivalent:
  \begin{enumerate}
    \item\label{structure2} There is a first order structure $M$ with domain
      $\Omega$ and relations of arity less than $\alpha$ whose automorphism group is $H$.
    \item\label{canonical2} $H$ is the automorphism group of the relational
      structure $(\Omega,O^H_{\bar a})_{\bar a \in\Omega^\alpha}$.
    \item\label{closed2} $H$ is closed in the $\alpha$-topology.
  \end{enumerate}
\end{prop}

Krasner's result that every subgroup of $G$ is $\Aut(M)$ for some relational structure $M$ when
relations of \textit{infinite} arity are allowed (first half of Theorem \ref{thm:kras1}) follows
from Proposition \ref{prop:generalizedtopology}. The $(|\Omega|\mathord+1)$-topology is the discrete
topology, \textit{i.e.} every set is open: If $g \in G$ and $\bar a$ is an enumeration of $\Omega$
then
$$G_{\bar{a},g\bar a}=\set{g}$$ is a basic open set.
Thus every set of permutations is the union of basic open sets, and therefore open; and therefore
also closed.
Hence, if relations of arity $|\Omega|$ are allowed, every group of permutations will be the
automorphism group of some relational (infinitary) structure.

\section{Quantifiers come into play} \label{section:extending}

\subsection{Second-order relations, $\lii$ and Krasner's correspondence}
We shall now extend Krasner's correspondence to second-order operations, in order to account for
quantifier extensions, which have been the traditional focus of the debates regarding logicality and
invariance. In this subsection, we state the corresponding generalization of Theorem \ref{thm:kras1}
and prove its first part.

A finite second-order relation $Q$ of type $(i_1,...,i_k)$ on $\Omega$ is a subset of $\power(\Omega
^{i_1}) \times ... \times \power(\Omega ^{i_k})$ for finite $k$ and finite $i_1$,...,$i_k$. A
second-order structure $\Q $ on a domain $\Omega$ is a set of (finite-ary) first-order and second-order
relations on $\Omega$. A permutation $g$ on $\Omega$ preserves a second-order relation $Q$ of type
$(i_1,...,i_k)$ if $(R^{i_1},...,R^{i_k}) \in Q$ iff $(gR^{i_1},...,gR^{i_k}) \in Q$ where
$R^{i_j}$, $1 \leq j \leq k$, are first-order relations of arities $i_j$. The mappings $\Aut$ and
$\Inv$ admit a straightforward generalization to the present setting: $\Aut(\Q)$ is the group of
permutations which preserve all first-order and second-order relations in $\Q$, and $\Inv(H)$, for
$H \subseteq G$, is the set of first-order and second-order relations which are preserved by all
permutations in $H$.

Given a second-order structure $\Q$, $\lii (\Q)$ is an interpreted language in the logic $\lii$: it
is the language whose signature matches the structure $\Q$ and whose predicate and quantifier
symbols are interpreted by the relations in $\Q$. The syntax and semantics for symbols interpreted
by second-order relations is familiar from generalized quantifier theory (see \cite{Lindstrom:66}).
As a case in point, let us recall what are the intended syntactic and semantic clauses for a
second-order relation of type $(2)$: Let $Q$ in $\Q$ be of type $(2)$, $\lii (\Q)$ is then equipped
with a matching quantifier symbol $\bar{Q}$. The syntactic clause for $\bar{Q}$ has it that if
$\phi$ is a formula of $\lii (\Q)$, so is $\bar{Q}x_1x_2~\phi$. The satisfaction clause for $\bar
{Q}$ is given by

\begin{gather*} \Omega, \Q \vDash \bar{Q}x_1x_2~\phi~~[\sigma] \\
\text{iff } \\ ||\phi||^\sigma_{x_1,x_2} \in Q  \end{gather*}

\noindent where $\sigma$ is an assignment over $\Omega$ and $$||\phi||^\sigma _x = \{ (a_1,a_2) \in
\Omega ^2~|~ \Omega, \Q \vDash \phi~~[\sigma [x_1:=a_1,x_2:=a_2]]\}.$$

We shall need to consider definability in $\lii (\Q)$ for a given $\Q$. Without loss of generality,
let $Q$ be a second-order relation of type $(2)$. We say that $Q$ is definable in $\lii (\Q)$ iff
there is a sentence $\phi_Q (\bar{R})$ in $\lii (\Q)$ expanded with a binary predicate symbol
$\bar{R}$ such that
\begin{gather*}
\Omega, \Q , R \vDash \phi_Q (\bar{R}) \\
\text{iff } \\ R \in Q
\end{gather*}
\noindent where $R$ is a binary first-order relation on $\Omega$ interpreting $\bar{R}$. The
$\lii$-closure of a second-order structure $\Q$ is the set of finite first-order and second-order
relations on $\Omega$ which are definable in $\lii (\Q)$.

We can now state the generalization of Theorem \ref{thm:kras1} to second-order structures and
automorphism groups thereof:

\begin{thm}\label{thm:kras2} Let $\Omega$ be a domain and $G=S_\Omega$ the symmetric group on $\Omega$, $H
\subseteq G$ any set of
  permutations and $\Q$ a second-order structure on $\Omega$.
  \begin{enumerate}
\item $Q \in \Inv(\Aut(\Q))$ iff $Q$ is definable in $\La_{\infty\infty}(\Q)$. The same holds for
relations $R$.
    \item $\Aut(\Inv(H))$ is the smallest subgroup of $G$ including $H$.

  \end{enumerate}
\end{thm}

We shall postpone the discussion and the proof of the second part of the theorem till the next
subsection. Regarding the first part, the proof that any relation definable in $\lii (\Q)$ is
invariant under the automorphisms of $\Q$ is a straightforward induction on the complexity of
formulas of $\lii (Q)$. What remains to be proven is the following lemma:

\begin{lemma}\label{lemma:Krasner}

Let $Q$ be a first-order or second-order relation on $\Omega$, if $Q \in \Inv(\Aut(\Q))$, then $Q$
is definable in $\lii (\Q)$.

\end{lemma}

\begin{proof}
For the sake of simplicity, we shall give the proof for $Q$ a second-order relation of type $(2)$.
So let $Q$ be invariant under the
automorphims of $\Q$. We first construct the sentence $\phi _Q (\bar{R})$ which defines it. Let us
fix an
enumeration $I \rightarrow \Omega$ of the elements in $\Omega$ and, given a set of variables $X$ of
cardinality
$|\Omega|$, an enumeration $I \rightarrow X$ of the variables in $X$. Let $\sigma$ be a partial
assignment over
$X$ defined by $\sigma (x_i)=a_i$ for $i \in I$.

For any relation $S$ on $\Omega$ and for any symbol $\bar{T}$ matching the type of $S$, let the
description
$\Delta_S(\bar{T})$ of $S$ by $\bar{T}$ be as follows. If $S$ is a first-order relation of arity
$k$,
$\Delta_S(\bar{T})$ is the conjunction of formulas of the form $\bar{T}x_{i_1}...x_{i_k}$ or $\neg
\bar{T}x_{i_1}...x_{i_k}$ which are satisfied in $\Omega,S$ under $\sigma$. Similarly, if $S$ is a
second-order
relation, say again of type (2), $\Delta_S(\bar{T})$ is the conjunction of formulas of the form
\begin{gather*} \bar{T}y_1,y_2 \underset{(i,j) \in K}{\bigvee} (y_1=x_i \land y_2 =
x_j)\end{gather*}
or
\begin{gather*} \neg \bar{T} y_1,y_2 \underset{(i,j) \in K}{\bigvee} (y_1=x_i \land y_2 = x_j)
\end{gather*}
for $K \subseteq I \times I$, which are satisfied in $\Omega,S$ under $\sigma$. Moreover, let $\psi$
be the
following formula

\begin{gather*}\underset{i,j \in I, i \neq j} {\bigwedge} x_i \neq x_j \land \forall y \underset{i
\in I}
{\bigvee} y=x_i \end{gather*}

\noindent that is, $\psi$ says that $\sigma$ is a bijection from $X$ onto $\Omega$.

$\phi _Q$ may now be written down as
\begin{equation*} \forall X~\Bigl( \Bigl(\underset{S \in \Q}{\bigwedge} \Delta_S(\bar{S}) \land
\psi\Bigr) \rightarrow
\underset{U \in Q}{\bigvee}
\Delta_U (\bar{R})\Bigr) \end{equation*}
\noindent where $\forall X$ is short for the infinite sequence of universal quantifiers $\forall
x_1,...,\forall
x_i,...$ quantifying over every variable in $X$.

We now prove that $\Omega, \Q , R \vDash \phi_Q (\bar{R}) \text{ iff } R \in Q$. From left to right,
if $\phi_Q
(\bar{R})$ is true in $\Omega, \Q , R$, we have in particular that
\begin{equation}\label{eq11}
\bigwedge_{S \in \Q}
\Delta_S(\bar{S}) \land \psi \rightarrow \bigvee_{U \in Q}
\Delta_U(\bar{R})
\end{equation}
is satisfied under $\sigma$. But, by construction, the premise of \eqref{eq11} is satisfied under
$\sigma$, hence there is a relation $U \in Q$ such that
$\Delta_U (\bar{R})$ is satisfied under $\sigma$ when $\bar{R}$ is interpreted by $R$. This forces
$U=R$, hence
$R \in Q$ as required.

From right to left, let $R$ be a first-order relation in $Q$ and $\sigma '$ a partial assignment on
$X$. Assume that $$\Omega, \Q , R \vDash
\bigwedge_{S \in \Q} \Delta_S(\bar{S}) \land \psi~~[\sigma'].$$ We need to show that
$\Omega, R \vDash \bigvee_{U \in Q} \Delta _U(\bar{R})~~[\sigma']$.

Since $\psi$ is satisfied under $\sigma '$, $\sigma '$ is a bijection, hence there is a permutation
$f$ on
$\Omega$ such that $f \circ \sigma ' = \sigma$. Moreover, since $\bigwedge_{S \in \Q}
\Delta_S(\bar{S})$ is satisfied under $\sigma'$, $f$ is an automorphism of $\Q$, hence $f$ preserves
$Q$.

By definition, $\Omega, fR \vDash \Delta _{fR}(\bar{R})~[\sigma]$. Since $\sigma = f \circ \sigma
'$, this implies
$\Omega, R \vDash \Delta _{fR}(\bar{R})~\sigma '$. Since $f$ preserves $Q$ and $R \in Q$, we also
have $fR \in
Q$. Therefore $\Delta _{fR}(\bar{R})$ is one the disjuncts in $\bigvee_{U \in Q} \Delta_U(\bar{R})$
and the formula is satisfied under $\sigma '$ as required.
\end{proof}

In \cite{McGee:96}, McGee proved that, on a given domain, an ``operation is invariant under all
permutations if
and only if it is described by some formula of $\lii$'' (p. 572). This result follows quite directly from what we have:\footnote{McGee's operations are actually defined as functions from
sequences of sets of assignments to sets of assignments. This difference does not matter in the present
context, since McGee's operation may be translated into second-order relations and back in such a way that
invariance under sets of permutations is preserved. } Consider the
empty
structure $\emptyset$. $\Aut(\emptyset)$ is the group $G$ of all permutations. By the first part of
Theorem
\ref{thm:kras2}, $\Inv(G)$ is the $\lii$-closure of $\emptyset$, which is exactly McGee's result. It
should be
noted that the proof of Lemma \ref{lemma:Krasner} itself is a rather straightforward generalization
of Krasner's
proof in \cite{Krasner:38}. The possibility to describe every subset of $\Omega^k$ by means of
formulas of $\lii$
is sufficient to handle the extension to second-order relations.

\subsection{A topology for second-order structures}

We shall now turn to the second half of Theorem \ref{thm:kras2} and ask when a group of permutations
is the automorphism group of a structure, in a context where second-order relations are allowed but
all relations are finitary. As earlier with Proposition \ref{prop:3}, we shall base our answer on
topological properties of groups of permutations. We will see that if we allow quantifiers of type
$( 2)$ then all groups of permutations are automorphism groups of some second-order structure.
However, the case is quiet different when we only allow monadic quantifiers, i.e., quantifiers of
type $( 1,1,\ldots,1 )$.
There is a natural topology such that exactly the closed groups are the automorphism groups of
monadic second-order structures.

We say a second-order structure is monadic if all relations are relations on the power set of
$\Omega$. A monadic second-order structure on $\Omega$ is nothing else than a first-order structure
on $\Omega$ together with a relational first-order structure on the power set $\power(\Omega)$ of
$\Omega$. Let $\G$ be the symmetric group on $\power(\Omega)$. The symmetric group $G=S_\Omega$ on $\Omega$
embeds in a natural way in $\G$, by the mapping $g \mapsto g^*$ where
$$g^*A=\set{ga | a \in A}.$$
Let $\G$ be equipped with the product topology.

The second-order topology on $G$ is defined by letting the basic open sets be
$$ G_{\bar A, \bar B} = \set{g \in G | g^*A_i=B_i \text{ for all $i$}}$$ where $\bar A=A_0,\ldots,
A_n$
and $\bar B=B_0,\ldots,B_n$ are subsets of $\Omega$. As is seen directly from the definition of the
basic
open sets the second-order topology on $G$ is the topology induced by the embedding ${}^\ast: G \to
\G$.

\begin{lemma}\label{lemma}
  The set $G^*=\set{g^* | g \in G}$ is a closed subgroup of $\G$.
\end{lemma}
\begin{proof}
Since ${}^*$ respects composition $G^*$ is a group. To see that it is closed
let $h \in S_\Omega$ be in the closure of $G^*$, i.e., for all $\bar A=A_0, \ldots,
  A_n$ there is $g \in G$ such that $hA_i=g^*A_i$ for $i \leq n$. Define
  $f\in G$ by $fa \in h\set{a}$, this uniquely determines $f$ since the set $h\set{a}$
  is a singleton set: There is $g \in G$ such that $h\set{a}=g^*\set{a}=\set{ga}$.

Assume now that $h \notin G^*$, then there is $A \subseteq \Omega$ such that $hA\neq f^*A$. We may
assume that there is $a \in hA \setminus f^*A$. For otherwise $h(A^c) \setminus f^*(A^c)\neq
\emptyset$ since the complement $(f^*A)^c$ of $f^*A$ is $f^*(A^c)$ and also that $h(A^c)=(hA)^c$
since there is $g\in G$ such that $g^*A=hA$ and $g^*(A^c)=h(A^c)$ (take $\bar A = A,A^c$).

Let $g \in G$ be such that $g^*A=hA$ and $g^*\set{f^{-1}a}=h\set{f^{-1}a}$ (take $\bar A =
A,\set{f^{-1}(a)}$).
Thus $gf^{-1}a=ff^{-1}a=a$ and so $f^{-1}a \in A$ contradicting that $a \notin f^*A$.
\end{proof}

From this lemma it follows that a set $A \subseteq G$ is closed in the second-order topology iff
$A^*$ is
closed in $\G$.

By using Proposition \ref{prop:3} and considering a second-order monadic  structure on the domain $\Omega$ as a first-order structure on $\power(\Omega)$, we can prove the characterization of automorphism
groups of monadic
second-order structures.

\begin{prop}
   Let $H$ be a subgroup of $G$. The following are equivalent:
  \begin{enumerate}
\item\label{sec_1} There is a second order monadic structure $\Q$ with domain $\Omega$ whose
automorphism group is $H$.
    \item\label{sec_2} $H$ is the automorphism group of the second-order structure
			$$\mathbb M = (\Omega,O^{H^*}_{\bar R})_{\bar R \in (\power(\Omega))^k }.$$
    \item\label{sec_3} $H$ is closed in the second-order topology.
  \end{enumerate}
\end{prop}
\begin{proof}
The implication \ref{sec_2} $\Rightarrow$ \ref{sec_1} is trivial.

To prove that \ref{sec_1} $\Rightarrow$ \ref{sec_3} let $\Q$ be a second-order structure with
$\Aut(\Q)=H$, let $M$ be the first-order part of $\Q$ and $\Q'$ the second-order part. By
Proposition \ref{prop:3} $\Aut(M) \subseteq G$ is closed. Now, $\Q'$ can be seen as a
\emph{first-order}
structure on the domain $\power(\Omega)$, and by applying Proposition \ref{prop:3} to this structure
on
$\power(\Omega)$ we get that $\Aut_{\power(\Omega)}(\Q') \subseteq \G$ is closed.\footnote{
$\Aut_{\power(\Omega)}(\Q')$ is the group of all permutations of $\power(\Omega)$ respecting the
quantifiers in $\Q'$. Compare this to $\Aut_{\Omega}(\Q')$, which is the set of all permutations of $\Omega$
respecting the quantifiers in $\Q'$.}
By the comment after Lemma \ref{lemma} $\Aut_{\Omega}(\Q') \subseteq G$ is closed since
$$\Aut_{\Omega}(\Q')^* = G^* \cap \Aut_{\power(\Omega)}(\Q'),$$
and both of these subgroups are closed. To conclude the argument it is enough to observe that
$\Aut(\Q) = \Aut(M) \cap \Aut(\Q')$ and that both of these groups are closed.

For the last implication \ref{sec_3} $\Rightarrow$ \ref{sec_2} it is enough to observe that $H^*
\subseteq
\G$ is closed and thus, by Proposition \ref{prop:3}, $H^*$ is the automorphism group of
$$(\power(\Omega),O^{H^*}_{\bar R})_{\bar R \in (\power(\Omega))^k }$$
	and that $g \in G$ respects $Q$ iff $g^* \in \G$ respects
$Q$, as a first-order quantifier on $\power (\Omega)$. Thus,
$H$ is the automorphism group of $\mathbb M$.
\end{proof}

To conclude the discussion of automorphism groups of monadic second-order structures we prove that
not all
subgroups of $G$ are closed in the second-order topology, implying that not all subgroups are
automorphism
groups of such structures which proves that the second part of Theorem \ref{thm:kras2} can not hold
if we
restrict $\Inv(\cdot)$ to only include monadic quantifiers. The following example of non-closed
subgroup comes
from Stoller \cite{Stoller:63}.

Assume that $\Omega$ is countable and
 $\leq$ well-orders $\Omega$. Say that a
permutation $p$ of $\Omega$ is a piece-wise order isomorphism if there are
partitions $(A_i)_{i\leq k}$ and $(B_i)_{i\leq k}$ of  $\Omega$ such that
$p$ restricted to $A_i$ is an order isomorphism from $A_i$ to $B_i$.
Let $H$ be the set of all such piece-wise order isomorphisms.

\begin{lemma}[{\cite{Stoller:63}}]
  $H$ is a proper dense subgroup of $G$.
\end{lemma}
\begin{proof}
For simplicity assume that $\Omega=\mathbb N$ and that the ordering $\leq$
coincides with the usual ordering of the natural numbers.

It is rather straight-forward to check that $H$ is a subgroup. That it is
proper follows from the following argument. Let $g(m)$ be
$(n + 1)^2 - (m +1-n^2)$ where $n^2\leq m<(n+ 1)^2$. The function $g$ is defined so that it
reverses arbitrarily long sequences of natural numbers: $g(k^2+a) = (k+1)^2 -a -1$ for $0 \leq a
\leq
2k$ which means that $g$ reverses the
sequence $k^2,k^2+1,\ldots, k^2+2k$ of length $2k+1$.
Assume that there are partitions $(A_i)_{i\leq k}$ and $(B_i)_{i\leq k}$ of $\mathbb N$ into $k$
blocks
such that $g$ restricted to $A_i$ is an order isomorphism from $A_i$ to $B_i$. By the pigeonhole
principle two elements of the
sequence $k^2,k^2+1,\ldots, k^2+2k$ of $2k+1$ is in the same block $A_i$. That contradicts the
assumption that $g$ is an order isomorphism from $A_i$ to $B_i$.

To see that $H$ is not closed take some $g \in G$. We prove that for any $\bar A$ there is
  $h \in H$ such that $g^*$ and $h^*$ agrees on $\bar A$, showing that $g$ is in
  the closure of $H$. Let the image of $\bar A=A_0,\ldots A_k$ under $g$ be $\bar
  B=B_0,\ldots,B_k$. For any subset $N$ of $[k]=\set{0,1,\ldots,k}$ let
  $$A_N=\bigcap
  \set{A_i | i \in N} \cap \bigcap \set{A_i^c | i \notin N},$$ and define
  $B_N$ similarily.
Then, for any $N \subseteq [k]$, $g$ maps $A_N$ to $B_N$; and also the $A_N$s are pairwise disjoint.  Clearly there is an $h \in H$ mapping the $A_N$s to the $B_N$s. It is now easy to see
  that $h$ also maps the $A_i$s to the $B_i$s. Thus, the closure of $H$ is $G$ and $H$ is
  dense in $G$.
\end{proof}

Let us now prove the second half of Theorem \ref{thm:kras2} by observing that the generalization of
the
second-order topology on $G$ where we regard the permuations of $G$ as permuting binary relations on $\Omega$ is the discrete topology and thus that any subgroup is closed in this topology:

\begin{prop}
  The topology on $G$ given by the basic open sets
$$ G_{\bar R, \bar S} = \set{g \in G | g^*(R_i)=S_i}$$ where $\bar R=R_0,\ldots,
R_n$
and $\bar S=S_0,\ldots,S_n$ are subsets of $\Omega^2$ is the discrete topology.
\end{prop}
\begin{proof}
  Let $<$ well-order $\Omega$ and let $g \in G$ be some permutation of $\Omega$.
  Let $R=\set{g(<)}$ be the image of $<$ under $g$. Then the basic open set
  $G_{<, R}$ is the singleton set $\set{g}$: If $h(<) = R$ and $h \in G$ then by transfinite
  induction on $x$ over $<$ we see that $h(x)=g(x)$ for every $x \in \Omega$ and thus that $h=g$.
\end{proof}

Let now $H$ be any subgroup of $G$, by the proposition $H$ is closed in this topology. By an
argument
similar to the case of monadic second-order structures $H^*$ is a closed subgroup of the symmetric
group
of $\power(\Omega^2)$. Thus by Proposition \ref{prop:3} there is a relational structure on
$\power(\Omega^2)$
whose automorphism group is $H^*$. This structure is a second-order structure on $\Omega$ having $H$
as its
automorphism group. Thus, we may conclude that any subgroup of $G$ is the automorphism group of a
second-order structure with quantifiers of type $(2,2\ldots,2)$; proving something slightly stronger
than
the second half of Theorem \ref{thm:kras2}.

\section{Invariance under similarities} \label{section:similarities}

In the present section, we propose an extension of Krasner's correspondence to languages without equality. Our motivations for doing so are twofold. First, we aim at greater generality: we shall show that the correspondence in Theorem \ref{thm:kras2} between groups of bijections and sets of first-order and second-order operations closed under definability is a special case of an even more encompassing correspondence. Second, invariance criteria in which functions are not required to be one-to-one have recently been proposed, see Feferman \cite{Feferman:99} and Casanovas \cite{Casanovas:07}:\footnote{Casanovas shows that several non-equivalent ways of defining of invariance appear when functions are allowed that are not one-one. We are concerned here with the what he calls `Feferman invariance', as defined by Feferman in \cite{Feferman:99} and that we take to be the most natural notion of invariance in that context. See \cite{Casanovas:07} for a detailed examination of the other options, as well as for characterization results for these other options.} the question is how much the situation does, or does not, change when invariance is thus liberalized.

The generalized correspondence is to hold between monoids of similarities and sets of operations closed under definability in \liim . A similarity $\pi$ on a domain $\Omega$ is simply a binary relation $\pi \subseteq \Omega \times \Omega$, such that for all $a \in \Omega$ there is a $b \in \Omega$ with $a \mathrel \pi b$ and for all $b \in \Omega$ there is $a \in \Omega$ such that $a \mathrel \pi b$. \liim~is the equality-free version of \lii . For a given signature, its formulas are all the formulas of \lii~that do not contain the equality symbol. In their recent model-theoretic studies of equality-free languages, Casanovas, Dellunde and Jansana have shown that structure preserving similarities play for first-order logic without equality the same role that isomorphisms play for first-order logic with equality, see \cite{Casanovas:96}.\footnote{In \cite{Casanovas:96}, Casanovas \& alii use the term `relativeness correspondence', whereas Feferman in \cite{Feferman:99} goes for `similarity relation'. Note also that, instead of similarity relations, one may work with surjective functions. Any relational composition of surjective maps or inverses of surjective maps is a similarity relation, and any similarity relation can be written as the
composition $f \circ g^{-1}$, where $f: C \to B$ and $g: C \to A$ are onto, see \cite{Casanovas:96} for more details.} The present extension of Krasner's correspondence pushes the analogy further.

Our extension comes with a twist. In \lii , on a fixed domain and with parameters, all subsets of the domain are definable. This is not the case in \liim . Generalized quantifiers may add expressive power that remains ineffective even when parameters are allowed. As a consequence, the proof method used in the proof of Theorem \ref{thm:kras2} breaks down. It relied on describing `from below' the action of invariant quantifiers by describing elementwise the sets and relations they apply to. Such a strategy is no longer available.

To see this, let us first define invariance under similarity. A relation $R \subseteq \Omega^k$ is invariant under the similarity relation $\pi$ on $\Omega$ if for all $\bar a \mathrel\pi \bar b$ we have $\bar a \in R$ iff $\bar b \in R$. Given a quantifier $Q$ over the domain $\Omega$ we say that $Q$ is invariant under $\pi$ if for all relations $R_1,\ldots, R_k, S_1,\ldots,S_k$ on $\Omega$ such that $R_i \mathrel\pi S_i$ we have $\langle R_1,\ldots R_k, \rangle \in Q$ iff $\langle S_1,\ldots, S_k \rangle \in Q$. Consider now, as an extreme case, the operations $\Q=\{\top,\bot,Q^{\text{E}} \}$, where $\top$ and $\bot$ are the two trivial unary relations\footnote{It is useful to assume that there are symbols in the language to express these, in order to guarantee that $\La_{\infty\infty}^-(\Q)$ is not empty.} defined by $\top=\Omega$, $\bot=\emptyset$ and $Q^{\text{E}}=\{ \{0,2,4,6,8,10,... \} \}$. One may show that $Q^{\text{O}}=\{ \{1,3,5,7,9,11,... \} \}$ is invariant under all similarities under which operations in $q$ are invariant. But in order to describe from below the action of $Q^{\text{O}}$, we would need to be able to define the set of odd numbers. As can easily be proven by induction, this is not possible in a language in which $\top$ and $\bot$ are the only relation symbols, even in the presence of $Q^{\text{E}}$.

Our alternative strategy will consist in showing that Krasner's correspondence holds \textit{when the extra expressive power is discarded}. More precisely, we will show that the desired definability and closure properties hold when attention is limited to the action of generalized quantifiers on definable sets and relations. Note however that even though $Q^{\text{O}}$ is not definable from below in $\liim (q)$, it is indeed definable full stop, simply by the formula $Q^{\text{E}}x~\neg Px$. Therefore, the question whether a version of Krasner's correspondence without a restriction to definable relations holds is still open.

Given a set of operations $q$ and a quantifier $Q$, let $Q^{\upharpoonright q}$ be the restriction of $Q$ to relations definable in $\lii (q)$ (with parameters):
$$Q^{\upharpoonright \Q} =\set{\langle R_1 \ldots R_k \rangle \in Q | R_i \text{ is definable in } \liim(q,a)_{a \in \Omega},~1\leq i \leq k }$$
Let $\Q^\upharpoonright$ be all the relations in $\Q$ and all the restrictions $Q^{\upharpoonright \Q}$ of quantifiers in $\Q$. It is easy to see that $\liim (\Q)$ and $\liim(\Q ^\upharpoonright)$ are elementary equivalent. In that sense, restricting quantifiers to their definable parts does not change the \textit{effective} expressive power of the logic.

Restriction to definable parts on the side of quantifiers will correspond to a similar restriction on the side of the invariance condition. Invariance will only be required with respect to relations that respect the underlying equivalence relation generated by the operations or the similarities under scrutiny. A set of operations $\Q$ gives us an equivalence relation $\sim_{\Q}$
corresponding to definability with parameters in $\La^-_{\infty,\infty}(\Q)$, that is $a \sim_{\Q}
b$ iff
\begin{equation}\label{eq1}
	\bigwedge_{\phi \in \La_{\infty\infty}^-(\Q)} \forall \bar
x (\phi(a,\bar x) \leftrightarrow \phi(b,\bar x)).
\end{equation}

Similarly, a set of similarities $\Pi$ generates an equivalence relation by the following condition:
$$ a \approx_\Pi b \text{ iff for all } \bar c\in \Omega^k \text{ there exists } \pi \in \Pi \text{
such that } a,\bar c \mathrel \pi b,\bar c, \text{ for a finite  } k.$$

Invariance is now parametrized by the equivalence relations $\sim$ we are considering.
Given a similarity $\pi$, a quantifier $Q$ is $\sim$-invariant under $\pi$ if for any
$\bar R,\bar S \in \power(\Omega^{k_1}) \times \ldots \times \power(\Omega^{k_l})$, all invariant under $\sim$, if
$\bar R \mathrel\pi \bar S$
then $ \bar R \in Q$ iff $ \bar S \in Q$. As before, a relation is $\sim$-invariant under $\pi$ if it is
invariant under $\pi$. This yields the main definitions of this section:
\begin{defin}\
\begin{itemize}
\item Let $\Inv(\Pi)$ be the set of all relations $R$ and quantifiers $Q$ on $\Omega$ which are
$\approx_\Pi$-invariant under all similarities in $\Pi$.
\item $\Sim(\Q)$ is the set of similarities $\pi$ such that all relations and quantifiers
in $\Q$ are $\sim_\Q$-invariant under $\pi$.
\end{itemize}
\end{defin}

Observe that when $\Pi$ is a group $H$ of permutations then $\Inv(\Pi)=\Inv(H)$, where $\Inv(H)$
uses the old definition of $\Inv$. This follows directly from the fact that $\approx_\Pi$ in this
case is equality.

In general, $\Sim(\Q)$ is not a group, because the similarities need not be invertible. The relevant closure properties for $\Sim(\Q)$ are the following. A set $\Pi$ of similarities is a \emph{monoid with involution} if it is closed under composition and taking converses.\footnote{The converse $R^{-1}$ of $R$ is such that $a\mathrel{R^{-1}} b $
iff $b \mathrel R a$.}. Moreover, $\Pi$ is \emph{full} if it is a monoid with involution, $\mathord\approx_\Pi \in \Pi$, and
closed under taking subsimilarities, i.e., such that if $\pi \in \Pi$ and $\pi' \subseteq \pi$ is a similarity then $\pi' \in \Pi$.

We are now ready to state Krasner's correspondence for the equality-free case, with the aforementioned restriction to the definable parts of quantifiers:

\begin{thm}\label{cor} Let $\Q$ be a set of operations and $\Pi$ a set of similarities, then
\begin{enumerate}
\item $Q \in \Inv(\Sim(\Q))$ iff $Q^{\upharpoonright q}$ is definable in
$\La_{\infty\infty}^-(\Q)$, and $R \in \Inv(\Sim(\Q))$ iff $R$ is definable in $\La_{\infty\infty}^-(\Q)$.
\item $\Sim(\Inv(\Pi))$ is the smallest full monoid including $\Pi$.
\end{enumerate}
\end{thm}

As a test case let $\Q$ include equality. Then $\sim_\Q$ is equality and $\Sim(\Q) = \Aut(\Q)$.
Theorem \ref{cor} says that a quantifier $Q$ is in $\Inv(\Aut(\Q))$ iff it is definable in
$\La_{\infty\infty}^-(\Q)$ which is the same as being definable in $\La_{\infty\infty}(\Q)$. Thus, part
one of Theorem \ref{thm:kras2} follows easily from Theorem \ref{cor}.
Also part two follows since if $H$ is a subgroup of the symmetric group of $\Omega$ then $\approx_H$
is nothing but equality and $\Aut(\Inv(H))$ is just $\Sim(\Inv(H))$ which is $H$ by (2) of Theorem
\ref{cor}.

In order to prove  Krasner's correspondence for definable relations, we shall need some technical machinery. We define two mappings: $\cdot \mathord/ \mathord\sim$ and $\cup$ which will operate on several different domains, but in a similar way.
Let $\sim$ be an equivalence relation on $\Omega$, $[a]=\set{b | a \sim
b}$ be the equivalence class of $a$, and $\Omega /\sim$ the set of all equivalence classes. Given $R \subseteq \Omega^k$
we define the relation $R/\mathord\sim$ on $\Omega/\mathord\sim$ by $$R/\mathord\sim =\set{\langle
[a_1], \ldots, [a_k] \rangle \in (\Omega/\mathord\sim)^k| \langle
a_1, \ldots, a_k \rangle \in R}.$$ If $R \subseteq
(\Omega/\mathord\sim)^k$ then $$\cup R = \set{  \langle
a_1, \ldots, a_k \rangle \in \Omega^k| \langle [a_1], \ldots, [a_k] \rangle \in
R}.$$ Similarily, given a quantifier $Q$ on $\Omega$, the quantifier $Q/\mathord\sim$ on
$\Omega/\mathord\sim$ is defined by
$$Q/\mathord\sim = \set{\langle R_1,\ldots,R_k\rangle | \langle \cup R_1,\ldots, \cup R_k\rangle \in
Q},$$ and given a quantifier $Q$ on $\Omega /\mathord\sim$, the quantifier $\cup Q$ on $\Omega$ is
defined by
$$\cup Q = \set{ \langle \cup R_1,\ldots,\cup R_k\rangle | \langle R_1,\ldots,R_k\rangle \in Q}.$$
Thus, given a set $\Q$ of operations on $\Omega$ and an equivalence relation $\sim$ on
$\Omega$, we get a set $\Q/\mathord\sim$ of operations on $\Omega/\mathord\sim$. Similarly
we can go from a set of operations on $\Omega /\mathord\sim$ to a set of operations on $\Omega$ by
the operation $\cup$.

\def\liimqpara{\liim(q,a)_{a \in \Omega}}

Since there are at most $2^{2^{|\Omega|}}$ many non-equivalent formulas with $|\Omega|$ free
variables the conjuction in \eqref{eq1} above can be bounded by that number, and thus \eqref{eq1}
is a formula of $\liim(q)$. We now get the following easy lemma.

\begin{lemma}\label{oneisenough}
For any $a$ and $b$ in $\Omega$, $a \sim_\Q b$ iff
$$\bigwedge_{\phi \in \La_{\infty\infty}^-(\Q)} \forall
x (\phi(a,x) \leftrightarrow \phi(b,x)), $$
where the conjuntion is only over $\phi$s with two free variables.
\end{lemma}
\begin{proof}
Follows directly by setting $\phi(x,y)$ to be the formula $x \sim_\Q y$ and $x$ to be $a$. Then we
have
$$ a \sim_\Q a \leftrightarrow a \sim_\Q b,$$
and thus that $a \sim_\Q b$.
\end{proof}

\begin{prop}\label{respect}
Let $\Q$ be a set of operations on $\Omega$.
\begin{enumerate}
\item  $R \subseteq \Omega^k$ is definable in $\liimqpara$ iff $R$ is invariant under $\sim_{\Q}$.
\item $Q^{\upharpoonright \Q}=\cup (Q/\sim_\Q)$
\end{enumerate}
\end{prop}

\begin{proof}
(1). The left to right implication is immediate. For the other implication assume $R$ is invariant under
$\sim_{\Q}$.
Let $\phi_R (x_1,...,x_n)$ be the formula
$$\bigvee_{\langle a_1,...,a_n \rangle \in R} \ \bigwedge_{1 \leq i \leq n} a_i \sim_q x_i$$
defining $R$. If $\bar{a} \in R$, $\Omega , \Q \vDash \phi (\bar{a})$ since $\sim_q$ is reflexive.
If $\Omega , \Q \vDash \phi (\bar{a})$, there is $\bar{b} \in R$ such that
$\bar a \sim_q \bar b$.
Implying that $\bar a \in
R$ since $R$ is invariant under $\sim_{\Q}$.

(2). By (1)  $R \subseteq \Omega^k$ is definable in $\liimqpara$ iff it is of the form $\cup R'$, where $R' \subseteq (\Omega/\sim_q)^k$. This holds iff $R=\cup (R/\sim_q)$.

Assume that $\bar R \in Q^{\upharpoonright \Q}$, then $R_i$ is definable in $\liimqpara$.
Thus $R_i = \cup (R_i/\sim_q)$, proving that $\bar R \in \cup (Q/\sim_q)$. On the other hand if $\bar R \in \cup (Q/\sim_q)$ then each $R_i$ satisfies $R_i=\cup (R_i/\sim_q)$ and thus they are all definable in $\liimqpara$. From the definition it also follows that $\cup (Q/\sim_q) \subseteq Q$ and thus $\bar R \in Q$.
 \end{proof}

\begin{lemma}\label{bijective}
Let $\Q$ be a set of operators and suppose $\pi \in \Sim(\Q)$. Then
\begin{enumerate}
\item $\pi/\mathord\sim_\Q$ is a permutation of $\Omega/\mathord\sim_\Q$,
\item if $R$ is invariant under $\sim_\Q$ then $R \mathrel\pi S$, where $S=\pi(R)$, and
\item if $\bar a, \bar b \in \Omega^\alpha$, $\bar a \mathrel \pi \bar b$ and $\phi(\bar x)$ is a formula in
$\La_{\infty\infty}^-(\Q)$ then
$\Omega, \Q \models \phi(\bar a) \leftrightarrow \phi(\bar b)$.
\end{enumerate}
\end{lemma}
\begin{proof}
Let $\sim$ be $\sim_\Q$.
	
	(1) To prove that $\pi/\mathord\sim$ is a function we prove that for all $a,a' \in \Omega $ and all
$b,b' \in \Omega$ such that $a \mathrel\sim a'$, $a \mathrel\pi b$ and $a' \mathrel\pi b'$ we have $b
\mathrel\sim b'$. This is proved by induction on the formulas in $\La^-_{\infty\infty}(\Q)$. In fact
we prove, by induction, the slightly stronger statement that if $\psi(x,\bar c)$ is a formula of
$\La^-_{\infty\infty}(\Q,\Omega)$ (observe that $\bar c$ can be an infinite string) then $\Omega,
\Q
\models \phi(b,\bar c)$ iff $\Omega, \Q \models \phi(b',\bar c)$.

For the base case take $R \in \Q$ and suppose $R(b,\bar c)$. Let $\bar d$ be such that $\bar d
\mathrel\pi \bar c$, then $R(a,\bar d)$ since $a,\bar d \mathrel\pi b,\bar c$. Now $R(a',\bar d)$
since $a \sim a'$ and then we have $R(b',\bar c)$ since $a',\bar d \mathrel\pi b',\bar c$.

For the induction steps negation, disjunction and existential quantification are
easy. Let us thus
take $Q \in \Q$ and suppose $\Omega, \Q \models Q\bar x\phi(b,\bar c,\bar x)$. Then $\llbracket \phi(b,\bar
c,\bar x)\rrbracket_{\bar x}^\Omega \in Q$ and
by the induction hypothesis $\llbracket \phi(b',\bar c,\bar x)\rrbracket_{\bar x}^\Omega
=\llbracket\phi(b,\bar c,\bar x)\rrbracket_{\bar x}^\Omega $ and thus $\Omega, \Q \models Q\bar x
\phi(b',\bar c,\bar x)$.

The proof of the bijective property is very similar and left to the reader.

(2) From (1) we have that if $\pi \in\Sim(\Q)$  and $R$ respects
$\sim$ then $R \mathrel\pi S$ iff $S=\pi(R)$. This is because if $\bar a \mathrel\pi \bar b$ and
$\bar b \in \pi(R)$ then there is $\bar a' \in R$ such that $\bar a' \mathrel \pi \bar b$ and so by
(1) $\bar a \sim \bar a'$ and since $R$ respects $\sim$, $\bar a \in R$.

(3) The proof is by induction on $\phi$. For the base case we have $\bar a \mathrel\pi \bar b$,
where $\pi \in \Sim(\Q)$,
and $R \in \Q$. Thus $R$ is invariant under $\pi$ and $R\bar a$ iff $R\bar b$.

For the induction step negation and (infinite) conjuntion is trivial. For the case with (an infinite
string of) existential quantifers assume that $\Omega,\Q \models
\exists \bar x \phi(\bar a, \bar x)$ is true and thus
there are $\bar c$ such that $\Omega, \Q \models \phi(\bar a,\bar c)$. Since $\pi$ is a similarity relation there are
$\bar d$ such that $\bar a,\bar c \mathrel\pi \bar b,\bar d$ and thus by the induction hypothesis we
have $\Omega, \Q \models \phi(\bar b,\bar d)$ and thus $\Omega, \Q \models \exists \bar x \phi(\bar b,\bar x)$.

We are left with the case of a generalized quantifier $Q \in \Q$. Suppose $\Omega, \Q \models
Q\bar x \phi(\bar a,\bar
x)$, i.e., that $R = \llbracket \phi(\bar a, \bar x)\rrbracket_{\bar x}^\Omega \in Q$.
Since $\pi \in \Sim(\Q)$, $R$ is invariant under
$\sim$ and by (2) $R \mathrel\pi S$ where $S = \pi(R)$, we have
$S \in Q$. Now, $\pi(R)=S$ is $\llbracket\phi(\bar b, \bar x)\rrbracket_{\bar x}^\Omega$ and thus
$\Omega, \Q \models Q\bar x \phi(\bar b,\bar x)$.
\end{proof}

%

Next we prove that $\approx_\Pi$ satisfies similar properties:

\begin{lemma}\label{bijective2}
Let $\Pi$ be a monoid of similarites with involution and $\pi \in \Pi$. Then
\begin{enumerate}
\item $\pi/\mathord\approx_\Pi$ is a permutation of $\Omega/\mathord\approx_\Pi$,
\item if $R$ is invariant under $\approx$ then $R \mathrel\pi S$, where $S=\pi(R)$, and
\item if also $\mathord\approx_\Pi \in \Pi$, $\bar a, \bar b \in \Omega^\alpha$,  $\bar a \mathrel \pi \bar b$ and $\phi(\bar x)$ is a formula in
$\La_{\infty\infty}^-(\Inv(\Pi))$ then
$\Omega, \Inv(\Pi) \models \phi(\bar a) \leftrightarrow \phi(\bar b)$.
\end{enumerate}
\end{lemma}
\begin{proof} Let $\approx$ be $\approx_\Pi$.

	(1) To prove that $\pi/\mathord\approx$ is a function we prove that for all $a,a' \in \Omega $ and
all $b,b' \in \Omega$ such that $a \mathrel\approx a'$, $a \mathrel\pi b$ and $a \mathrel\pi b'$ we
have $b \mathrel\approx b'$.
This follows directly from $\Pi$ being closed under composition and
taking inverses.
The proof of the bijective property is similar and left to the reader.

(2) If $\bar a \mathrel\pi \bar b$ and $\bar b \in \pi(R)$ then there is $\bar a' \in R$ such that
$\bar a' \mathrel \pi \bar b$ and so by (1) $\bar a \approx \bar a'$ and since $R$ is invariant
under $\approx$, $\bar a \in R$.

(3) The proof is by induction on $\phi$. For the base case we have $\bar a \mathrel\pi \bar b$,
where $\pi \in \Pi$,
and $R \in \Inv(\Pi)$. Thus $R$ is invariant under $\pi$ and $R\bar a$ iff $R\bar b$.

For the induction step negation and (infinite) conjuntion is trivial. For the case with (an infinite
string of) existential quantifiers assume that $\Omega, \Inv(\Pi) \models
\exists \bar x \phi(\bar a, \bar x)$ is true and
thus there are $\bar c$ such that $\Omega, \Inv(\Pi) \models \phi(\bar a,\bar c)$. Since $\pi$ is a similarity relation there
are $\bar d$ such that $\bar a,\bar c \mathrel\pi \bar b,\bar d$ and thus by the induction
hypothesis we have $\Omega, \Inv(\Pi) \models \phi(\bar b,\bar d)$ and $\Omega, \Inv(Pi)\models
\exists \bar x \phi(\bar b,\bar x)$.

We are left with the case of a generalized quantifier $Q \in \Q$. Suppose $Q\bar x \phi(\bar a,\bar
x)$, i.e., that $R = \llbracket\phi(\bar a,
\bar x)\rrbracket_{\bar x}^\Omega \in Q$.
If $R$ is invariant under $\approx$ we are finished since then by (2) $R
\mathrel\pi S$ where $S = \pi(R)$, we have
$S \in Q$. Now, $\pi(R)=S$ is $\llbracket\phi(\bar b, \bar x)\rrbracket_{\bar x}^\Omega$ and thus
$\Omega, \Inv(\Pi)\models Q\bar x \phi(\bar b,\bar x)$. To
see that $R$ is invariant under $\approx$ take $\bar c \in R$ and $\bar d$ such that $\bar c \approx
\bar d$. There is $\pi \in \Pi$ such that $\bar c \mathrel\pi \bar d$ and then since $\mathord\approx_\Pi
\in \Pi$, $\pi' = \pi \circ \mathord\approx \in \Pi$ and
$\bar a,\bar c \mathrel{\pi'} \bar a,\bar d$, and thus $\Omega, \Inv(\Pi) \models
\phi(\bar a,\bar c)$ implies that $\Omega, \Inv(\Pi)\models \phi(\bar
a,\bar d)$ by the induction hypothesis and we have $\bar d \in R$.
\end{proof}

In fact $\approx_\Pi$ and $\sim_\Q$ are very much related as the next proposition shows.

\begin{prop}\label{allisgood} Let $\Q$ be a set of operators and $\Pi$ a monoid with involution,
then
\begin{enumerate}
\item if  $\mathord\approx_\Pi \in \Pi$, then $\approx_\Pi$ is the same relation as $\sim_{\Inv (\Pi)}$, and
\item $\sim_\Q$ is the same as $\approx_{\Sim(\Q)}$.
\end{enumerate}
\end{prop}

\begin{proof}
(1) First, $\mathord\approx_\Pi \subseteq \mathord\sim_{\Inv (\Pi)}$. Assume $a \approx_\Pi b$. For
any $c \in \Omega$, there is a $\pi \in \Pi$ such that $ac~\pi~bc$, hence, by (3) in Lemma
\ref{bijective2} and since $\Pi$ is full, for every formula $\phi(x,y)$ in
$\La_{\infty\infty}^-(\Inv(\Pi))$, we have $\Omega , \Inv(\Pi) \vDash \phi(a,c)$ iff $\Omega ,
\Inv(\Pi) \vDash \phi (b,c)$. By Lemma \ref{oneisenough}, this suffices for $a \sim_{\Inv (\Pi)} b$.

Second, $\mathord\approx_\Pi \supseteq \mathord\sim_{\Inv (\Pi)}$. Assume $a \sim_{\Inv (\Pi)} b$.
For any $\bar c \in \Omega ^k$, define a relation $H^\Pi_{a\bar c}$ of arity $k+1$ by $H^\Pi_{a\bar
c} = \{ \bar d \in \Omega ^{k+1}~|~\text{there is } \pi \in \Pi \text{ such that } a \bar
c~\pi~\bar d \}$. Check that $H^\Pi_{a\bar c}$ is invariant under $\Pi$. Assume that $\bar d~ \pi
~\bar d'$ for some $\pi \in \Pi$. If $\bar d \in H^\Pi_{a\bar c}$, there is a $\pi ' \in \Pi$ with
$a,\bar c~\pi '~\bar d$. Since $\Pi$ is closed under composition, $\pi ' \circ \pi \in \Pi$, and
then $a , \bar c~\pi '\circ \pi~\bar d'$ implies $\bar d' \in H_{a\bar c,\Pi}$. If $\bar d' \in
H^\Pi_{a\bar c}$, there is a $\pi ' \in \Pi$ with $a,\bar c~\pi '~\bar d '$. Since $\Pi$ is closed
under composition and taking converses, $\pi ' \circ \pi^{\smallsmile} \in \Pi$, and then $a ,\bar
c~\pi ' \circ \pi ^{\smallsmile}~\bar d$ implies $\bar d \in H^\Pi_{a\bar c}$. Therefore $H^\Pi_{a\bar
c}$ is the interpretation of a predicate symbol $P$ in the language
$\La_{\infty\infty}^-(\Inv(\Pi))$. By definition, $\Omega , \Inv(\Pi) \vDash
P(a,\bar{c})$. Since $a \sim_{\Inv(\Pi)} b$, this implies $\Omega ,
\Inv(\Pi) \vDash P(b,\bar{c})$, hence there is a $\pi \in \Pi$ such that
$a\bar{c}~\pi~b\bar{c}$.

(2) First, $\sim_\Q \supseteq \approx_{\Sim(\Q)}$. Assume $a \approx_{\Sim(\Q)} b$. For any $c \in
\Omega$, there is a $\pi \in \Pi$ such that $ac~\pi~bc$, hence, by (3) in Lemma \ref{bijective}, for
every formula $\phi(x,y)$ in $\La_{\infty\infty}^-(\Q)$, we have $\Omega , \Q
\vDash \phi(a,c)$ iff $\Omega , \Q \vDash \phi (b,c)$. By Lemma \ref{oneisenough}, this
suffices for $a \sim_{\Inv (\Pi)} b$.

Second, $\sim_\Q \subseteq \approx_{\Sim(\Q)}$. Assume $a \sim_\Q b$. Let $\bar c \in \Omega^k$, we
need to find a $\pi \in \approx_{\Sim(\Q)}$ with $a\bar c~\pi b \bar c$. It is sufficient to show
that $H^{\Sim(\Q)}_{a\bar c}$ is definable in $\La_{\infty\infty}^-(\Q)$ by some formula
$\varphi$, because then $\Omega, \Q \vDash \varphi(a,\bar c)$ and $a \sim_\Q b$ implies
$\Omega, \Q \vDash \varphi(b,\bar c)$, giving the needed $\pi$.

$H^{\Sim(\Q)}_{a\bar c}$ is invariant under $\Sim(\Q)$ since $\Sim(\Q)$ is closed under
compositions. By
Proposition \ref{respect}, $H_{a\bar c}^{\Sim(\Q)}$ is definable in
$\La_{\infty\infty}^-(\Q,\Omega)$
if $H_{a\bar c}^{\Sim(\Q)}$ is invariant under $\sim_\Q$. Thus, it is sufficient to show that $\sim_\Q
\in \Sim(\Q)$. Let $R$ be an $n$-ary relation in $\Q$, $d_1 ... d_n \in R$ and
$d_1 ... d_n~\sim_\Q~d'_1 ... d'_n$. $d_1 ... d_n \in R$ and $d_1\sim _\Q d'_1$,
hence $d'_1d_2...d_n \in R$, and by
repeating this reasoning, $d'_1 ... d'_n \in R$. Let $Q$ be a quantifier in $\Q$, $\bar R$ and $\bar
S$ sequences of relations on $\Omega$ invariant under $\sim_\Q$ and of a type appropriate for $Q$.
We need that if $\bar R \sim_\Q \bar S$, then $\bar R \in Q$ iff $\bar S \in Q$. It is sufficient
to note that if $R$ is invariant under $\sim_\Q$ and $R \sim_\Q S$, then $R =S$. Let $\bar d \in
R$ and $\bar e$ such that $\bar e \sim _\Q \bar d$. $\bar e \in R$ since $R$ is invariant under
$\sim_\Q$, hence since $R \sim _\Q S$, $\bar d \in S$. Similarly, let $\bar d \in S$ and $\bar e$
such that $\bar e \sim_\Q \bar d$. Since $R \sim_\Q S$, $\bar e \in R$, and then $\bar e \sim_\Q
\bar d$ and the fact that $R$ is invariant under $\sim_\Q$ implies $\bar d \in R$.
\end{proof}

\begin{lemma} $\Sim(\Q)$ is a full monoid.
\end{lemma}
\begin{proof}
Let $\Pi = \Sim(\Q)$. To check that $\Pi$ is closed under taking converses is trivial, and by using
(2) of Lemma \ref{bijective} we can easily see that $\Pi$ is also closed under compositions.

We prove that if $\pi \in \Pi$ and $\pi' \subseteq \pi$ is a similarity relation then $\pi' \in
\Pi$. It is trivial to
check that every relation of $\Q$ is invariant under $\pi'$. For the case of quantifiers let $R
\mathrel{\pi'} S$ be such that
$R$ and $S$ are invariant under $\sim_\Q$. Then $R \mathrel\pi S$ because if not we may assume,
without loss of generality, there are
$\bar a \mathrel\pi \bar b$, $\bar a \in R$ and $\bar b \notin S$. There are $\bar c$ such that
$\bar a \mathrel{\pi'} \bar c$ and thus, since both $\bar a \mathrel\pi \bar b$ and $\bar a
\mathrel\pi \bar c$ we have by (1) in Lemma \ref{bijective} that $\bar b \sim_\Q \bar c$
contradicting that $S$ is invariant under $\sim_\Q$.
\end{proof}

Given a set of operations $\Q$ we now know that $\sim_\Q$ is the same as $\approx_{\Sim(\Q)}$ so
when working with $\Inv(\Sim(\Q))$ we can use $\sim_\Q$ instead of $\approx_{\Sim(\Q)}$; this will
implicity be used in the following proposition.

\begin{prop}\label{prop:aut}
Let $\Q$ be a set of operators and $\sim$ be $\sim_\Q$. We have
\begin{enumerate}
\item $\Sim(\Q)/\mathord\sim$ is $\Aut(\Q/\mathord\sim)$, and
\item $\Inv(\Sim(\Q))$ is the set of all quantifers $Q$ such that $Q/\mathord\sim$ is in
$\Inv(\Aut(\Q/\mathord\sim) )$.
\end{enumerate}
\end{prop}
\begin{proof}
	(1)
	For the left-to-right inclusion let $\pi \in \Sim(\Q)$. By  Lemma \ref{bijective}, $f=\pi/\mathord\sim$ is a
bijection on $\Omega/\mathord\sim$ and if $Q \in \Q$ then $R
\in Q/\mathord\sim$ iff $\cup R \in Q$ iff $\pi(\cup R) \in Q$ iff $\cup f(R)
\in Q$ iff $f(R) \in Q/\mathord\sim$.

For the other inclusion let $f \in \Aut(\Q/\mathord\sim)$.
Define $a \mathrel\pi b$ iff $f([a])=[b]$. Then $\pi$ is a similarity
relation on $\Omega$ and $f = \pi /\mathord\sim$.
Let $Q \in \Q$ and $R$ is invariant under $\sim$ then $R \in Q$ iff
$R/\mathord\sim \in
Q/\mathord\sim $ iff $f(R/\mathord\sim) \in Q/\mathord\sim $ iff $\cup
f(R/\mathord\sim) \in Q$. But $\cup f(R/\mathord\sim) = \pi(R)$. We
have proved that $R \in Q$ iff $\pi(R) \in Q$, and thus $Q$ is
invariant under $\pi$.

(2) Take $Q \in \Inv(\Sim(\Q))$ and $f \in \Aut(\Q/\mathord\sim)$.
By (1) choose $\pi \in \Sim(\Q)$ such that
$f=\pi/\mathord\sim$. Then $R
\in Q/\mathord\sim$ iff $\cup R \in Q$ iff $\pi(\cup R) \in Q$ iff
$\pi(\cup R)/\mathord\sim \in Q/\mathord\sim$, but $ \pi(\cup R)/\mathord\sim =
f(R)$.

On the other hand take $Q$ such that $Q/\mathord\sim \in
\Inv(\Aut(\Q /\mathord\sim))$, $R$ invariant under $\sim$ and $\pi
\in \Sim(\Q)$, then $f= \pi/\mathord\sim$ is in
$\Aut(\Q /\mathord\sim)$. Thus $R \in Q$ iff $R/\mathord\sim
\in Q/\mathord\sim$ iff $f(R/\mathord\sim)\in Q/\mathord\sim$ iff
$\cup f(R/\mathord\sim)\in Q$, but $\cup f(R/\mathord\sim)= \pi(R)$.
\end{proof}

From this proposition we get what was announced as the main theorem of this section:

\begin{reptheorem}{cor} Let $\Q$ be a set of operations and $\Pi$ a set of similarities, then
\begin{enumerate}
\item $Q \in \Inv(\Sim(\Q))$ iff $Q^{\upharpoonright q}$ is definable in
$\La_{\infty\infty}^-(\Q)$, and $R \in \Inv(\Sim(\Q))$ iff $R$ is definable in $\La_{\infty\infty}^-(\Q)$.
\item $\Sim(\Inv(\Pi))$ is the smallest full monoid including $\Pi$.
\end{enumerate}
\end{reptheorem}

\begin{proof}
(1) By Proposition \ref{prop:aut}, $Q \in \Inv(\Sim(\Q))$ iff
$Q/\mathord\sim \in \Inv(\Aut(\Q/\mathord\sim) )$. Thus,
by Theorem \ref{thm:kras2}, $Q/\mathord\sim$ is
definable in $\La_{\infty\infty}(\Q/\mathord\sim)$ by a sentence $\phi$.
Let $\phi'$ be the translation of $\phi$ into $\La_{\infty\infty}^-(\Q)$ where $=$
is replaced by $\sim$.

By induction of formulas we see that for all $\bar a \in \Omega$ and all $\psi(\bar x)$ in
$\La_{\infty\infty}(\Q/\mathord\sim)$:
$$\Omega/\mathord\sim \models \psi([\bar{a}]) \text{ iff } \Omega \models \psi'(\bar a),$$
where, again, $\psi'$ is the translation of $\psi$  into $\La_{\infty\infty}^-(\Q)$.

It follows that $\phi'$ defines $\cup (Q/\mathord\sim)$ and since $\sim$ is definable in
$\La_{\infty\infty}^-(\Q)$ we have what we wanted.

(2) Follows by the same argument as in (1).

(3) We prove that for a full monoid $\Pi$ we have that $\Sim(\Inv(\Pi))=\Pi$. The general theorem
then follows from the fact that every $\Sim(\Q)$ is a full monoid.

The inclusion $\Pi \subseteq \Sim(\Inv(\Pi))$ is trivial. For the other includion let $H=
\Pi/\mathord\approx$. It is a group of permutations by Lemma \ref{bijective2}. By straightforward
checking we see that
 $$\Inv(\Pi) / \mathord\approx = \Inv(H).$$

From (1) in Proposition \ref{prop:aut} we get that $\Sim(\Inv(\Pi))/\mathord\approx =
\Aut(\Inv(\Pi)/\mathord\approx)$ and thus that this is equal to $\Aut(\Inv(H))$, which by Theorem
\ref{thm:kras2} is nothing else than $H$. Thus,
$$\Sim(\Inv(\Pi))/\mathord\approx = \Pi/\mathord\approx$$
and by the assumption that $\Pi$ is full we get that $\Sim(\Inv(\Pi)) = \Pi$: If $\pi \in
\Sim(\Inv(\Pi))$ then there is $\pi' \in \Pi$ such that $\pi'/\mathord\approx =
\pi/\mathord\approx$, and then $\pi \subseteq (\mathord\approx \mathrel\circ \pi' \mathrel\circ
\mathord\approx)$, and thus $\pi \in \Pi$.
\end{proof}

\section{Concluding remarks}\label{section:final}

\subsection{Invariance under permutation as a logicality criterion}

Let us first go back to languages with equality and groups of permutations. The main benefit of a Krasnerian approach to McGee's result is to show that the correspondence between purportedly logical operations on a domain and the symmetric group on that set is a special case of the general duality which holds between sets of relations and groups of permutations, when sets of relations consist in first-order and second-order relations closed under definability in
$\lii$. The light thus shed on McGee's result is ambiguous, depending on how the closure condition is looked at. One might insist that groups of permutations are consequently shown to be a sure guide to characterizing sets of relations. When $\lii$ provides the logical context, \textit{any} relevant structure may be characterized in terms of invariance under permutation, and therefore the structure consisting of all and only logical relations on a given domain, in particular, is thus characterizable. Now which group of permutations should be the one giving us all (and only the) logical relations? Since logical notions have generality as their distinguishing feature, we should go for the largest possible group, and take all permutations: this is the gist of the generality argument. Since logic is formal and does not make differences between objects, we should require insensitivity to any possible way of switching objects, and take again all permutations: this is the gist of the formality argument. Of course, several steps are in need of further justification. Maybe identifying `being a logical notion' and `being a general notion' is too quick, because generality remains underspecified. Or maybe `being formal' and `not making distinctions between objects' are not quite the same -- see MacFarlane \cite{MacFarlane:00} for a thorough discussion of these issues. But our goal here is not to challenge the relevance of the conceptual analysis in terms of generality or indifference to objects. Our point is rather to suggest that these arguments may be deemed to be on the right track \textit{only if we know that logical relations are characterizable in terms of invariants of a group of permutations}. Once this characterizability is granted, both arguments go through, and nicely converge.

Seeing McGee's result as a special case of a more encompassing Krasnerian correspondence reveals $\lii$ to be a built-in feature of the framework. Definability in $\lii$ is the closure condition on sets of relations which corresponds to working with groups of permutations, rather than with some other kind of transformations. As soon as we choose to explicate logicality (either qua generality or qua formality) in terms of groups of permutation, we let infinitary syntax sneak in. Now the generality and formality arguments tell us which group of permutations we should pick, if we are to single out logical relations by picking one, but they do not tell us that we should single out logical relations by picking a group of permutations to start with. Thus, the fact that the infinitary logic $\lii$ ends up being a full-blown logic on this account of logicality might be deemed an artificial effect of the invariance under permutation framework. McGee himself took his result to be an argument in favor of invariance for permutation as a logicality criterion, ``inasmuch as every operation on the list is intuitively logical'' (\cite{McGee:96}, p.567), the list being the list of basic operations of $\lii$. But this presupposes that having reasons to count conjunction as logical automatically counts as having reasons to count infinitary conjunction, and similarly for quantification.

Blaming invariance under permutation for its infinitary upshot is no new criticism (see \cite{Feferman:99}, \cite{Bonnay:08} and \cite{Feferman:10}), but, in the light of the generalized Krasnerian correspondence,  this infinitary burden is seen to be the very consequence of the use of permutations in order to demarcate the extensions that are admissible for logical constants.

\subsection{Invariance under similarity as a logicality criterion}

Invariance under similarities has been put forward as an alternative logicality criterion by Feferman \cite{Feferman:99}.\footnote{Feferman now favors mixed approches in which invariance only provides necessary conditions for logicality, see \cite{Feferman:10} and \cite{Feferman:11}.} The shift from permutations to similarities makes for greater generality, which is crucial for the generality argument. It also provides a way to escape the commitment of permutations to cardinalities. Every cardinality quantifier, no matter how wild (`there are at least $\aleph_{17}$ many') is permutation invariant. Shifting to invariance under similarities expells these wild creatures from the paradise of logical quantifiers.\footnote{Feferman had another reason to favor similarities, namely the fact that they can relate domains of different sizes, and hence force a greater homogeneity for the extensions of quantifiers across domains. We do not deal with across-domain invariance here.}

Feferman shows that an operation is definable in first-order logic without equality just in case it is definable in the $\lambda$-calculus from homomorphism invariant\footnote{Homomorphism invariance is equivalent to invariance under similarities.} \textit{monadic} quantifiers and asks whether ``there is a natural characterization of the homomorphims invariant propositional operations in general, in terms of logics extending the predicate calculus'' (\cite{Feferman:99}, p. 47). The question is all the more pressing that the syntactic restriction to monadic quantifiers lacks a proper justification. With reference to an unpublished proof given in \cite{bonnay06}, it has been claimed that invariance under homomorphism corresponds to definability in pure $\liim$ (see \cite{Bonnay:08} and \cite{Feferman:10}). The proof, however, was faulty, breaking down because of undefinable subsets. Theorem \ref{cor} provides a correct generalization. However, it does not answer Feferman's exact question, since the correspondence we eventually get only holds with respect to the action of quantifiers on definable sets. What we know by Theorem \ref{cor} is that invariance under homomorphism corresponds to definability in pure $\liim$ when quantifiers are restricted to their action on definable subsets of the domain. Whether the original claim was correct is still unknown.

The restriction to definable sets notwithstanding, Theorem \ref{cor} shows that the shift from permutations to similarities or homorphisms does not change the picture when it comes to the infinitary explosion of invariance criteria: infinitary syntax is every bit as much hardwired in invariance under similarities as it is in invariance under permutations. Moreover, invariance criteria do not seem well suited to help us adjudicate the status of equality as a logical constant or not. As highlighted by the close parallelism between Theorem \ref{thm:kras2} and Theorem \ref{cor}, allowing identity at level of relations exactly corresponds to banning non-injective functions at the level of invariance conditions. Whether there are better grounds on which to adjudicate the issue, or whether it should be considered as primarily a matter of convention remains to be seen.

\subsection{Further perspective}

To conclude, we would like to mention two lines of further research. The first one is somewhat tangential to the main project. We have seen that automorphism groups of monadic second order structures are characterized as closed groups in a way similar to the first-order case. What do these closed groups tell us about the structures they come from? The question is whether remarkable properties of monadic second order structures can be reduced to properties of automorphism groups -- a famous example in the first-order case is the Ryll-Nardzewski theorem which says that being $\aleph_0$ categorical is such a property.

The second line of research concerns further inquiries into the duality between sets of transformations and sets of relations closed with respect to definability in a given logic. Theorem \ref{cor} leaves as an open question whether the restriction to the definable parts of quantifiers is mandatory. More generally, this duality could be extended to other kinds of transformations and less demanding closure conditions. In \cite{Barw73}, Barwise provides such a result for first-order relations with respect to semi-automorphisms (automorphisms between substructures which are part of a back and forth system) and closure under definability in $\lio$. This suggests generalizing again to second-order relations. Another tempting move would consist in getting to $\loo$ by suitably weakening the conditions on semi-automorphisms. Finally, the connection between properties of the logical syntax and properties of the transformations still remains to be fully understood. It would be nice to achieve even greater generality by connecting explicitly abstract properties of transformations (e.g. the depth of Ehrenfeucht-Fra\"{\i}ss\'{e} games) and properties of the syntax for the logic providing the closure condition.

\bibliography{krasner}
\bibliographystyle{asl}

\end{document}